\makeatletter
\@namedef{subjclassname@2020}{\textup{2020} Mathematics Subject Classification}
\makeatother
\def\bbuildrel#1_#2^#3{\mathrel{\mathop{\kern 0pt#1}\limits_{#2}^{#3}}}

\def\NN{\mathbb N}

\def\PP{\mathbb P}

\def\RR{\mathbb R}

\def\bs{\bigskip}
\def\ms{\medskip}
\def\ss{\smallskip}

\def\w{\thinspace\hbox{\hsize 14pt \rightarrowfill}\thinspace}

\def\0{\hbox{$\emptyset$}}

\def\I{\hbox{$\mathcal I$}}

\def\M{\hbox{\rm MGR}}

\def\s{\hbox{$\sigma$}}

\def\sub{\subseteq}
\def\G{\hbox{$\mathcal G$}}
\def\M{\hbox{$\mathcal M$}}

\def\Ra{\Rightarrow}
\def\Lra{\Leftrightarrow}

\def\K{\mathscr{K}}

\def\M{\mathscr{M}}
\def\G{\mathscr{G}}

\def\U{\mathscr{U}}

\documentclass[a4paper,12pt]{amsart}

\usepackage{amssymb,enumerate,mathrsfs}
\usepackage[utf8]{inputenc} % to daje polskie litery z użyciem AltGr

\theoremstyle{plain}

\newcommand{\co}{\mathfrak c}

\newcommand{\baire}{\NN^{\NN}}
\newtheorem{theorem}{Theorem}[section]

\newtheorem{lemma}[theorem]{Lemma}

\newtheorem{proposition}[theorem]{Proposition}

\newtheorem{remark}[theorem]{Remark}
\newtheorem{cor}[theorem]{Corollary}

\numberwithin{equation}{section}

\begin{document}

%\title{On two consequences of CH established by Sierpi\' nski}
\title{On two consequences of CH established by Sierpi\' nski. II} %-- additional results}

\author{R. Pol and P. Zakrzewski}
\address{Institute of Mathematics, University of Warsaw, ul. Banacha 2, 02-097 Warsaw, Poland}
\email{pol@mimuw.edu.pl, piotrzak@mimuw.edu.pl}


\date{March 5, 2024}
%wersja posłana do ArXiv

	%concentrated set}

\begin{abstract}

We continue a study of the relations between two consequences of the Continuum Hypothesis discovered by Wacław Sierpiński, concerning uniform continuity of continuous functions and uniform convergence of sequences of real-valued functions, defined on subsets of the real line of cardinality continuum.

\end{abstract}

\maketitle
\section{Introduction}\label{sec:1} 
In \cite{C_8_vs_C_9} we studied the following two consequences of the Continuum Hypothesis (CH) distinguished
by Wacław Sierpiński in his classical treaty \textit{Hypoth\`{e}se du continu} \cite{Si}   (the notation is taken from \cite{Si}):

\begin{enumerate}
	\item[${C}_8$] There exists a continuous function $f:E\to \RR$, $E\sub\RR$, $|E|=\co$, not uniformly continuous on any uncountable subset of $E$.
	
	\item[$C_9$] There is a sequence of functions $f_n:E\to\RR$, $E\sub\RR$, $|E|=\co$, converging pointwise but not converging uniformly on any uncountable subset of $E$.
\end{enumerate}

Sierpiński \cite{si-2} checked that $C_8$ implies ${C}_9$.
% but he did not discuss 
The status of the converse implication remains unclear. 
%despite the fact that
Let us notice, however, that in \textit{Topology I} by Kuratowski \cite{kur-I}, footnote (3) on page 533 suggests that the two  statements are in  fact equivalent. 

In \cite{C_8_vs_C_9} we  considered the following stratifications of statements $C_8$ and $C_9$ for  uncountable cardinals  $\kappa\leq\lambda\leq\co$: 
\begin{enumerate}
	\item[$C_8(\lambda,\kappa)$] 
	There exists a set $E\sub \RR$ 
	%$E\sub\RR,\baire, X$? 
	of cardinality $\lambda$ and a continuous function $f:E\w \RR$, which is not uniformly continuous on any subset of $E$ of cardinality $\kappa$.
	
	\item[$C_9(\lambda,\kappa)$] There exists a set $E\sub \RR$ of cardinality $\lambda$ (equivalently: for any set $E\sub \RR$ of cardinality $\lambda$)
	%$E\sub\RR,\baire, X$? 
	and	there is a sequence of functions $f_n:E\to\RR$,  converging on $E$ pointwise but not converging uniformly on any subset of $E$ of cardinality $\kappa$.
\end{enumerate}

%Clearly, statements $C_i$  are $C_i(\co,\aleph_1)$ in our notation, and $C_i$ implies $C_i(\lambda,\kappa)$ for all uncountable cardinals  $\kappa\leq\lambda\leq\co$, $i=8,9$.

In particular, we proved in \cite{C_8_vs_C_9} that:

\begin{itemize}
	
	\item $C_8(\co,\co)\Lra C_9(\co,\co)$, and each of these statements is equivalent to the assertion $\mathfrak{d}=\co$, provided that the cardinal $\co$ is regular.% (cf. Theorem \ref{C_8(c,c)_iff_C_9(c,c)}),

	\item $C_8(\aleph_1,\aleph_1)\Lra C_9(\aleph_1,\aleph_1)$,
	and each of these statements is equivalent to the assertion $\mathfrak{b}=\aleph_1$.
	% (cf. Theorem \ref{C_8(w_1,w_1)_iff_C_9(w_1,w_1)}).
\end{itemize}
Here $\mathfrak{d}$ and $\mathfrak{b}$ denote, as usual, the smallest cardinality of a dominating and, respectively, an unbounded family in $\baire$ corresponding to the ordering of eventual domination $\leq^*$ (cf. \cite{ha}).  

\ms 
An important role in our considerations was played by the notion of a K-Lusin set  (cf.  \cite{ba-ha}) which we  extended  declaring that an uncountable subset $E$ of a Polish space $X$ is a {\sl $\kappa$-$K$-Lusin set in $X$},  $\aleph_1\leq \kappa\leq\co$, if $|E\cap K|<\kappa$  for every  compact set $K\sub X$. We proved in \cite{C_8_vs_C_9} that $C_9(\lambda,\kappa)$ is equivalent to the statement that there is a Polish space $X$ and a $\kappa$-K-Lusin set of cardinality $\lambda$ in $X$.

\ms 

In this note we present  additional results related to the
subject 
%matter
 of \cite{C_8_vs_C_9}. Most of them were earlier  announced in \cite[Section 4]{C_8_vs_C_9}.

\ms

In Sections \ref{sec:2}, \ref{sec:3} and \ref{sec:4} we investigate $C_8$-like phenomena 
%for functions between metric spaces other that $\RR$
in a more general setting. We are interested in two closely related problems ($X$ and $Y$ are fixed separable metric spaces).

\ss 

{\sl Problem 1.} Is the  existence of a set $E\sub X$ 
of cardinality $\lambda$ and a continuous function on $E$ with values in  $Y$, which is not uniformly continuous on any subset of $E$ of cardinality $\kappa$, related to either $C_8(\lambda,\kappa)$ or $C_9(\lambda,\kappa)$? 
 
\ss
{\sl Problem 2.} Does the existence of  a set $E\sub X$ of cardinality $\lambda$ and a continuous function on $E$ with values in  $Y$, which is not uniformly continuous on any subset of $E$ of cardinality $\kappa$, imply that there also exists such a function on $E$ with values in $\RR$?

\ss
Concerning Problem 1, we observe that the existence of  a separable metric space $X$ of cardinality $\lambda$, a metric space $Y$, and a continuous function on $X$ with range in $Y$, which is not uniformly continuous on any subset of $X$ of cardinality $\kappa$, already implies that there exists a $\kappa$-K-Lusin set of cardinality $\lambda$ in some Polish space and, consequently, that $C_9(\lambda,\kappa)$ holds true (cf. Proposition \ref{K-Lusin_from_any_example}). 

Conversely, $C_9(\lambda,\kappa)$ implies that there exists a $\kappa$-K-Lusin set $E$ of cardinality $\lambda$ in $\PP$, the set of irrationals of the unit interval $I=[0,1]$,
 such that for every non-$\sigma$-compact Polish space $Y$ there is a continuous function on $E$ which is not uniformly continuous on any subset of $E$ of cardinality $\kappa$ (cf. Theorem \ref{C_9}). 

On the other hand, if $Y$ is compact, then the existence of  a set $E\sub I$ of cardinality $\lambda$, and a continuous function $f:E\to Y$, which is not uniformly continuous on any subset of $E$ of cardinality $\kappa$, implies 
%(and is, therefore, equivalent to)
 $C_8(\lambda,\kappa)$ (cf. Theorem \ref{C_8}).  

\ss

Concerning Problem 2, we show  that if a set $E$ in the Hilbert cube $I^\NN$ is  zero-dimensional and there exists a continuous function on $E$ with range in any uncountable compact metric space $Y$, not uniformly continuous on any  subset of $X$ of cardinality $\kappa$, then there is also such a function with range in any  uncountable compact metric space $Z$, and in particular, in $I$ (cf. Theorem \ref{theorem 1}).

On the other hand, assuming CH, we prove the existence of a set  $E\sub I^\NN$ of cardinality $\co$ such that 
there is a continuous function $f:E\w I^\NN$, which is not uniformly continuous on any subset of $E$ of cardinality $\co$ but each continuous function $g:E\to \RR$ is constant on a subset of $E$  of cardinality $\co$. The construction of a witnessing pair $E$ and $f$ falls under a general scheme, described in Section \ref{sec:4}, of 
%constructing
constructions of $C_8$-like examples 
%with the help of
based on a generalization of the notion of a
$\kappa$-$K$-Lusin set. 

\ms 

Section \ref{sec:5} is 
%devoted to a characterization of complete metrizability of separable metrizable spaces (cf. \ref{characterization_of_completeness}) and its generalization to the effect
a slight departure from the topic but it is closely related to a reasoning of Sierpiński concerning $C_9$. We shall show that a Hausdorff space $X$ is \v{C}ech-complete and Lindel\"{o}f  if and only if there is a sequence $f_0\geq f_1\geq\ldots$ of continuous functions $f_n:X\to I$ converging pointwise to zero but not converging uniformly on any closed non-compact set in $X$ (cf. Theorem \ref{Cech_complete}). The existence of such a sequence for any Polish space $X$ was  a key step in proving that statement $C_9(\lambda,\kappa)$ is equivalent to the existence of a  $\kappa$-K-Lusin set $E$ of cardinality $\lambda$ in $\PP$ (cf.  \cite[Theorem 2.3]{C_8_vs_C_9}).

\bs

In this note $\PP$ always denotes the set of irrationals of the unit interval $I=[0,1]$.
%moze zmienic na I??
 It is homeomorphic to the Baire space $\baire$, the countable product of the set of natural numbers $\NN=\{0,1,2,\ldots\}$ with the discrete topology (cf. \cite{ke}).

\section{Mappings into non-compact spaces and $C_9$}\label{sec:2}

We start with a general observation.

\begin{proposition}\label{K-Lusin_from_any_example}
	If there exist  a separable metric space $(X,d_X)$ of cardinality $\lambda$, a metric space $(Y,d_Y)$, and a continuous function on $X$ with range in $Y$, which is not uniformly continuous on any subset of $X$ of cardinality $\kappa$, then there exists a $\kappa$-K-Lusin set of cardinality $\lambda$ in some Polish space and, consequently, $C_9(\lambda,\kappa)$ holds true. 
\end{proposition}

\begin{proof}
 Let $f: X\to Y$ be a continuous function which is not uniformly continuous on any subset of $X$ of cardinality $\kappa$. The function $f$ extends (cf. \cite[Theorem 3.8]{ke}) to a continuous function  $\tilde{f}:G\to \hat{Y}$ over a $G_\delta$-set $G$ (with $X\sub G$) in the Polish completion $(\hat{X},\hat{d_X})$ of $(X,d_X)$ into the completion $(\hat{Y},\hat{d_Y})$ of $(Y,d_Y)$. Now, if $K\sub G$ is compact, then $\tilde{f}$ is uniformly continuous on $K$, hence $f$ is uniformly continuous on $K\cap X$. By the choice of $X$, we have that $|K\cap X|<\kappa$, which shows that $X$ is a $\kappa$-K-Lusin set of cardinality $\lambda$ in $G$. By  \cite[Theorem 2.3]{C_8_vs_C_9}, this proves  $C_9(\lambda,\kappa)$.
	
\end{proof}

On the other hand, statement $C_9(\lambda,\kappa)$ already implies (and in view of Proposition \ref{K-Lusin_from_any_example}, is equivalent to) the existence of $C_9$-like example for functions with values in arbitrary non-$\sigma$-compact Polish spaces. 

\begin{theorem}\label{C_9}
	For any uncountable cardinals $\kappa\leq\lambda\leq\co$	the following are equivalent:
	\begin{enumerate}
		\item $C_9(\lambda,\kappa)$,
		% i.e., there is a sequence of functions %$f_n:E\w \RR$, $n=1,2,\ldots$, defined on a set $E$ of cardinality $\lambda$ 	$f_n:\lambda\w \RR$,		 which converges to zero pointwise but does not converge uniformly on any set of cardinality $\kappa$,
		\smallskip
		
		%\item there exist a set $E\sub \RR$ 
		%%$E\sub\RR,\baire, X$? 
		%of cardinality $\lambda$ and a continuous function $f:E\w \RR^\NN$, which is not uniformly continuous (with respect to any complete metric on $\RR^\NN$) on any subset of $E$ of cardinality $\kappa$,
		
		\item there exist a set $E\sub \RR$ 
		%$E\sub\RR,\baire, X$? 
		of cardinality $\lambda$, a non-$\sigma$-compact Polish space $Y$, and  a continuous function $f:E\to Y$ which is not uniformly continuous (with respect to arbitrary complete metric on $Y$) on any subset of $E$ of cardinality $\kappa$. 
		%Moreover, any $\kappa$-K-Lusin set $E$ of cardinality $\lambda$ in $\PP$ has the above property with respect to any non-$\sigma$-compact Polish space.
	\end{enumerate}
		
	Moreover, any $\kappa$-K-Lusin set $E$ of cardinality $\lambda$ in $\PP$ has the  property expressed in (2) with respect to any non-$\sigma$-compact Polish space $Y$.	
	%\smallskip
		
	%	\item 	there exist  a set $E\sub \RR$ of cardinality $\lambda$, a metric space $Y$, and a continuous function on $E$ with range in $Y$, which is not uniformly continuous on any subset of $E$ of cardinality $\kappa$.
		
	%	\smallskip 
		
	%	\item there exist  a separable metric space $(X,d_X)$ of cardinality $\lambda$, a metric space $(Y,d_Y)$, and a continuous function on $X$ with range in $Y$, which is not uniformly continuous on any subset of $X$ of cardinality $\kappa$.
		
		\begin{proof}
			$(1)\Ra (2).$ By Theorem \cite[Theorem 2.3]{C_8_vs_C_9}, statement $C_9(\lambda,\kappa)$ implies that there is a  $\kappa$-K-Lusin set $E$ of cardinality $\lambda$ in $\PP$. 
			
			Let $Y$ be an arbitrary non-$\sigma$-compact Polish space. Let $h:\PP\w Y$ be a homeomorphic embedding of $\PP$ onto a closed subspace $h(\PP)$ of $Y$ (cf. \cite[Theorem 7.10]{ke}). We will show that $E$ together with $f=h|E$ have the required properties. 
			
			To that end, let us fix a set $A\sub E$ with $|A|=\kappa$. Then $\bar{A}$, the closure of $A$ in $I$, is not contained in $\PP$ since otherwise $\bar{A}$ would be a compact set in $\PP$, intersecting $E$ on a set of cardinality $\kappa$. 
			So let us pick $a_k\in A$, $k\in\NN$, such that $\lim_{k\to\infty}a_k=a$ and $a\in I\setminus\PP$.
			
			Now, if $f$ were uniformly continuous on $A$ with respect to a complete metric $d$ on $Y$, $f$ would take Cauchy sequences in $A$ to Cauchy sequences in $Y$. In particular, the sequence $(f(a_n))_{n\in\NN}$ would be Cauchy in $Y$, hence $\lim_{k\to\infty}f(a_k)=z$ for some $z\in Y$. This, however, is not the case: since the set $\{a_0, a_1,\ldots\}$ has no accumulation point in $\PP$, the set $\{f(a_0), f(a_1),\ldots\}$ has no accumulation point in $h(\PP)$ and hence also in $Y$, as $h(\PP)$ is closed in $Y$. 
			\ss 
			
			The implication $(2) \Ra (1)$ follows immediately from Proposition \ref{K-Lusin_from_any_example}.

		%	$(4)\Ra (1).$ Let $f: X\to Y$ be a continuous function on a separable metric space $(X,d_X)$ of cardinality $\lambda$ into a metric space $(Y,d_Y)$, which is not uniformly continuous on any subset of $X$ of cardinality $\kappa$. The function $f$ extends (cf. \cite[Theorem 3.8]{ke}) to a continuous function  $\tilde{f}:G\to \hat{Y}$ over a $G_\delta$-set $G$ (with $X\sub G$) in a completion $(\hat{X},\hat{d_X})$ of $(X,d_X)$ into a completion $(\hat{Y},\hat{d_Y})$ of $(Y,d_Y)$. Now, if $K\sub G$ is compact, then $\tilde{f}$ is uniformly continuous on $K$, hence $f$ is uniformly continuous on $K\cap X$. By the choice of $X$, we have that $|K\cap X|<\kappa$, which shows that $X$ is a $\kappa$-K-Lusin set of cardinality $\lambda$ in $G$. By  \cite[Theorem 2.3]{C_8_vs_C_9}, this proves  $C_9(\lambda,\kappa)$,  $G$ being Polish as a $G_\delta$-set in  $(\hat{X},\hat{d_X})$.
			
		\end{proof}

\end{theorem}

\section{Mappings into compact spaces and $C_8$}\label{sec:3} 

The results of the previous section show that $C_8$-like statements for functions with values in non-$\sigma$-compact Polish spaces are actually equivalent to statement $C_9$. 
%It turns out that
Apparently,  the situation changes when we consider functions with values in compact spaces. As the following  result  shows, if a  set $E\sub I^\NN$ is  zero-dimensional and there exists a continuous function on $E$ with range in any uncountable compact metric space $Y$, not uniformly continuous on any  subset of $X$ of cardinality $\kappa$, then there is also such a function with values in $\RR$, witnessing that $C_8(\lambda,\kappa)$ holds true.

	\begin{theorem}\label{theorem 1}
		Let $E$ be a 
  %separable, - to juz wynika z tego, ze to jest podzbior metrycznej zwartej \hat{X} 
  zero-dimensional subset of a compact 
  metric space $X$.
  %(equivalently, $E$ is a zero-dimensional space equipped with a totally bounded metric). 
		%of cardinality $\lambda$.
		If there exists a continuous function on $E$ with range in a compact metric space $Y$, not uniformly continuous on any  subset of $E$ of cardinality $\kappa$, then there is also such a function with range in 
  %$[0,1]$.
  the Cantor ternary  set $C$ in $I$. Consequently, for any uncountable compact metric space $Z$, there is also such a function with values in $Z$.
	\end{theorem}
	
	\begin{proof}
		Let $h: E\to Y$ be a continuous function not uniformly continuous on any set of cardinality $\kappa$. 
		%Let $Y=\bigcup\limits_{n\in\NN} Y_n$ with each $Y_n$ a compact subspace of $Y$.  
		
		Using the compactness of $Y$, let as fix a  sequence $(K_n,L_n)_{n\in\NN}$ of pairs of disjoint compact sets in $Y$ such that for any pair $(K,L)$ of disjoint compact sets in $Y$, there is $n$ with $K\sub K_n$ and $L\sub L_n$. 
		
		%For each $n\in\NN$, we let $C_n=h^{-1}(K_n)$, $D_n=h^{-1}(L_n)$, and using the fact that $X$ is zero-dimensional and separable, we choose a continuous function $u_n: X\to \{0,1\}$ taking on $C_n$ value 0 and on $D_n$ value 1, cf. \cite[Theorem 6.2.7]{eng}.

  For each $n\in\NN$, we let $C_n=h^{-1}(K_n)$, $D_n=h^{-1}(L_n)$, and using the fact that $E$ is zero-dimensional and separable, we choose a continuous function $u_n: E\to \{0,2\}$ taking on $C_n$ value 0 and on $D_n$ value 2, cf. \cite[Theorem 1.2.6]{eng-2}.
		
		We shall show that the function $f:E\to I$ defined by the formula
		$$
		%f(x)=\sum_{n=0}^\infty 4^{-n}u_n(x)
  f(x)=\sum_{i=0}^\infty \frac{1}{3^{i+1}}u_i(x),
		$$
		which takes values in the  Cantor ternary set $C$, is not uniformly continuous on any  subset of $E$ of cardinality $\kappa$.
		
		To that end, let us fix a set $A\sub E$ with $|A|=\kappa$.
		% Let $(\hat{X},\hat{d})$ be a completion of $(X,d)$; since $d$ is totally bounded, $(\hat{X},\hat{d})$ is compact.
		
		We shall first make the following observation. Let $a\in \bar{A}$, the closure of $A$ in $X$. Then for any $n$, since $u_n$ takes on $A$ values 0 or 2 only, the oscillation
		$$
		\inf\{\hbox{\rm diam} (u_n(A\cap V)):\ V\ \hbox{is a neighbourhood of $a$ in}\ X\}
		$$
		of $u_n|A$ at $a$ is either 0 or 2.
		
		Let us note that  $h|A: A\to Y$ cannot be extended to a continuous (hence, by the compactness of $\bar{A}$, uniformly continuous) function $\bar{h}:\bar{A}\to Y$, since otherwise the function $h|A=\bar{h}|A$ would itself be uniformly continuous, contrary to the the fact that $|A|=\kappa$. Consequently, there must be closed disjoint sets $K$, $L$ in $Y$ such that the closures of $(h|A)^{-1}(K)$ and  $(h|A)^{-1}(L)$ in $X$ meet, cf. \cite[Theorem 3.2.1]{eng-1}. Then for an $n$ with $K\sub K_n$ and $L\sub L_n$, we infer that $\overline{C_n\cap A}\cap  \overline{D_n\cap A}\neq \emptyset$, so let us fix $a\in \overline{C_n\cap A}\cap  \overline{D_n\cap A}$.
		
		It follows that $u_n|A$ has the oscillation at $a$ equal to 2 and let us assume that $n$ is the smallest index with this property. The oscillation of each of the functions $u_0|A,\ldots, u_{n-1}|A$ is then equal to 0, hence we can find a neighbourhood $V$ of $a$ in $X$ such that all these functions have constant values on $A\cap V$.
		
		Let us pick $x_k,\ y_k\in A\cap V$, $k\in\NN$, such that $\lim_{k\to\infty}x_k=\lim_{k\to\infty}y_k=a$ and for %every $k$, $u_n(x_k)=0$, $u_n(y_k)=1$. Then 	
		%$$	\begin{matrix} f(y_k)-f(x_k)   =  4^{-n}+\sum\limits_{m=n+1}^\infty 4^{-m}(u_m(y_k)-u_m(x_k)) & \geq  \\  4^{-n}-\sum\limits_{m=n+1}^\infty 4^{-m}\geq \frac{2}{3}4^{-n}.	\end{matrix}$$
				%It follows that the oscillation of $f|A$ at $a$ is at least $\frac{2}{3}4^{-n}$. In effect, $f|A$ has no continuous extension over $\bar{A}$, which means that $f$ is not uniformly continuous on $A$.
every $k$, $u_n(x_k)=0$, $u_n(y_k)=2$, hence	
		$$	\begin{matrix} f(y_k)-f(x_k)   =  \frac{2}{3^{n+1}} +\sum\limits_{i=n+1}^\infty \frac{1}{3^{i+1}}(u_i(y_k)-u_i(x_k)) & \geq  \\ \frac{2}{3^{n+1}} -\sum\limits_{i=n+1}^\infty \frac{2}{3^{i+1}}= \frac{1}{3^{n+1}}.	\end{matrix}$$

		It follows that the oscillation of $f|A$ at $a$ is at least $3^{-n-1}$. In effect, $f|A$ has no continuous extension over $\bar{A}$, which means that $f$ is not uniformly continuous on $A$.

  \ss 

  Finally, if $(Z,d)$ is an arbitrary compact metric space and   $e: C\to Z $ is a homeomorphic embedding of $C$ into $Z$, then since $e^{-1}$ is uniformly continuous,  the  function $e\circ f:E\to Z$ has desired properties.

	\end{proof}
	
	%(for a set $B\sub Y$, $\hbox{\rm diam}_{d} (B)$ stand for the diameter of $B$ with respect to $d$).
	
As an immediate corollary we obtain the following equivalent form of statement $C_8(\lambda,\kappa)$.

\begin{theorem}\label{C_8}

	For any uncountable cardinals $\kappa\leq\lambda\leq\co$ if  the cofinality of $\lambda$ is uncountable, then
		the following are equivalent:
	\begin{enumerate}
		\item $C_8(\lambda,\kappa)$,
	
	%	\smallskip
	%	\item there exist a set $E\sub \RR$ 
		%$E\sub\RR,\baire, X$? 
	%	of cardinality $\lambda$ and a continuous function $f:E\w [0,1]$, which is not uniformly continuous on any subset of $E$ of cardinality $\kappa$.
		
		\smallskip
%		\item 	there exists a set $E\sub \RR$ 	of cardinality $\lambda$ such that for every uncountable compact metric space $Y$ there is a continuous function on $E$ with range in $Y$, which is not uniformly continuous on any subset of $E$ of cardinality $\kappa$,

%%%Nie wiem, czy to prawda; 
		
%\item 	for every uncountable compact metric space $Y$,	there exists a set $E\sub \RR$ of cardinality $\lambda$ and a continuous function on $E$ with range in $Y$, which is not uniformly continuous on any subset of $E$ of cardinality $\kappa$,
%\item there exist a set $E\sub \RR$ 
%$E\sub\RR,\baire, X$? 
%of cardinality $\lambda$ and a continuous function $f:E\w [0,1]^\NN$, which is not uniformly continuous on any subset of $E$ of cardinality $\kappa$.

%	\smallskip

%\item 	there exist  a set $E\sub \RR$ of cardinality $\lambda$, a compact metric space $Y$, and a continuous function $f:E\to Y$, which is not uniformly continuous on any subset of $E$ of cardinality $\kappa$.

%\item 	there exist  a set $E\sub [0,1]$ of cardinality $\lambda$, an uncountable compact metric space $Y$, and a continuous function $f:E\to Y$, which is not uniformly continuous on any subset of $E$ of cardinality $\kappa$.

\item 	there exists  a set $E\sub \RR$ of cardinality $\lambda$, a compact metric space $Y$ and  a continuous function $f:E\to Y$, which is not uniformly continuous on any subset of $E$ of cardinality $\kappa$. 

%Moreover, if a set $E\sub [0,1]$ has the desired property for at least one uncountable compact metric space, then it also has this property for any other such space. 

%\item there exist a zero-dimensional metric space $(X,d)$ of cardinality $\lambda$ with a compatible completely bounded metric $d$, a compact metric space $Y$,  and a continuous function $f:X\to Y$, not uniformly continuous on any  subset of $X$ of cardinality $\kappa$. 

  %%Nie wiem, jak przejsc w dziedzinie do podzbioru R (z zachowaniem metryki)
 
	\end{enumerate}

Moreover,
% if  the cofinality of $\kappa$ is uncountable, then
 any set $E\sub I$ that witnesses $C_8(\lambda,\kappa)$ 
 %together with a function $f:E\to I$
 for some continuous function  from $E$ to $I$  has the  property expressed in (2)
with respect to any uncountable compact metric space.
\end{theorem}

\begin{proof}
	$(1)\Ra (2).$ If $f:E\w \RR$ is a continuous function on a set $E\sub \RR$ 
		of cardinality $\lambda$, which is not uniformly continuous on any subset of $E$ of cardinality $\kappa$, then since the cofinality of $\lambda$ is uncountable, by shrinking $E$, if necessary, we may assume that the range of $f$ is contained in a closed  interval $Y$ of length 1.
		
		 If additionally $E\sub I$ and $f:E\to I$,  then since $E$ contains no non-trivial interval, it is zero-dimensional, and Theorem \ref{theorem 1}  applies.
   
\ss 

 $(2)\Ra (1).$  
  Now let $f:E\to Y$ be  a continuous function with values in a compact metric space $Y$, which is not uniformly continuous on any subset of $E$ of cardinality $\kappa$. We may again assume that $E$ is contained in a closed interval $X$. Then, $E$ being zero-dimensional, it is enough to apply the final part of the assertion of Theorem \ref{theorem 1} to $Z=I$.

\end{proof}

By the results of \cite{C_8_vs_C_9} (cf. Section \ref{sec:1}), Theorems \ref{C_9} and \ref{C_8} lead to the following corollary.

\begin{cor}
	If either $\kappa=\lambda=\co$ or $\kappa=\lambda=\aleph_1$, then the following are equivalent:
	\begin{enumerate}
		\item 	There exist a set $E\sub \RR$ 
		%$E\sub\RR,\baire, X$? 
		of cardinality $\lambda$, a non-$\sigma$-compact Polish space $Y$, and  a continuous function $f:E\to Y$ which is not uniformly continuous (with respect to any complete metric on $Y$) on any subset of $E$ of cardinality $\kappa$.
		
		\item There exist  a set $E\sub \RR$ of cardinality $\lambda$, a compact metric space $Y$ and  a continuous function $f:E\to Y$, which is not uniformly continuous on any subset of $E$ of cardinality $\kappa$.

	\end{enumerate}
\end{cor}

\section{Constructing $C_8$-like examples from $\K$-Lusin sets}\label{sec:4}

 Throughout this section we assume that $G$ is an uncountable $G_\delta$-set in a  compact metric space $X$.

	Given a collection $\K$ of  compact sets  in 
	%$\hat{X}$
	$X$ containing all singletons, we say that 
 an uncountable subset $E$ of  $G$ is a {\sl $\kappa$-$\K$-Lusin set in $G$}, where  $\aleph_1\leq \kappa\leq\co$, if $|E\cap K|<\kappa$  for every   set $K\in \K$. In particular, if $\K=K(G)$, the collection of all compact sets in $G$, then  a $\kappa$-$\K$-Lusin set in $G$ is just a $\kappa$-$K$-Lusin set in $G$.

 \ms

The  $C_8$-like examples, presented later in this section
% with the help
by means of $\K$-Lusin sets,  are based on the following observation, applied also in the proofs of \cite[Theorems 3.8 and 3.9]{C_8_vs_C_9}. 

\begin{proposition}\label{C_8 from calK-Lusin}
	%Let $(X,d)$ be a Polish non-compact metric space where the  metric $d$ is totally bounded and let $(\hat{X},\hat{d})$ be a 
	%	completion  of $(X,d)$. 
	
	% Let $G$ be a $G_\delta$ set in $X$ and 
	 Let $\varphi:{X}\w Y$ be  a continuous map onto a compact metric space $Y$ such that $\varphi|G$ is a homeomorphism onto $\varphi(G)$. Let ${\K}$ be a collection of  compact sets  in ${X}$ such that whenever $A\sub G$ and $\varphi|A$ extends to a homeomorphism over $\bar{A}$, the closure of $A$ in ${X}$, then $\bar{A}\in \K$.
	 
		If $T\sub G$ is a $\kappa$-$\K$-Lusin set in $G$ of cardinality $\lambda$, where $\aleph_1\leq \kappa\leq\lambda\leq\co$,
		%such that $|T|=\lambda$ and $|T\cap A|<\kappa$ for %every $A\in {\K}$,
		 then letting $E=\varphi(T)$ and $f=\varphi^{-1}|E:E\to {X}$, we obtain a continuous function on a set of cardinality $\lambda$, which is not uniformly continuous on any subset of $E$ of cardinality $\kappa$.
		\begin{proof} %We follow closely the proof of \cite[Theorems 3.8]{C_8_vs_C_9}.
				Aiming at a contradiction, assume that $f|B$ is uniformly continuous (with respect to any metric compatible with the topology of $Y$) on a set $B\sub E$ of cardinality $\kappa$ and let $A=f(B)=\varphi^{-1}(B)$. Then, since $\varphi|A:A\to B$ is also uniformly continuous, the function $\varphi|A$ extends to a homeomorphism over $\bar{A}$ (cf. \cite[Theorem 4.3.17]{eng-1}). Consequently, 
				$\bar{A}\in \K$. This, however, is impossible, since on one hand we have $|T\cap \bar{A}|<\kappa$, $T$ being a $\kappa$-$\K$-Lusin set in $X$,  but on the other hand $A\sub T\cap \bar{A}$ has cardinality $\kappa$.  
		\end{proof}
	
\end{proposition}

\subsection{Zero-dimensional spaces}
Throughout this subsection	let us additionally assume that the (compact metric) space $X$ is zero-dimensional.
% and let  $(\hat{X},\hat{d})$ be its compact, zero-dimensional completion, cf. \cite[1.3.16]{eng-2}.

\ss 

A  proof of the following fact is given in \cite[Lemma 4.2]{mazurkiewicz} (it is based on an idea similar to that in \cite[proof of Lemma 5.3]{eng-2}).
\begin{lemma}\label{compactification of zero-dim} For any $G_\delta$-set $G$ in $X$
		there is a continuous map $\varphi:{X}\to Y$ onto a compact metric space $Y$, such that
	 $\varphi|G$ is a homeomorphism onto $\varphi(G)$, $\varphi(X\setminus G)\cap \varphi(G)=\emptyset$ and the set $Y\setminus \varphi(G)$ is countable. 
	 %(a simple argument to this effect is given in \cite[Lemma 4.2]{mazurkiewicz}).
\end{lemma}

Various $C_8$-like examples 
%concerning functions from  zero-dimensional spaces into compact metric spaces (which, according to Theorem \ref{theorem 1} lead to ``real" $C_8$ examples)
 could be constructed with the help of 
%Proposition \ref{C_8 from calK-Lusin},
 Lemma \ref{compactification of zero-dim} and the following observation. By a {\sl $\sigma$-ideal on X} we mean a collection $\I$ of Borel sets in $X$, closed under taking Borel subsets and
 countable unions of elements of $\I$, and  containing all singletons.

\begin{proposition}\label{zero-dim examples from ideals}
 Let  $\I$ be a \s-ideal on ${X}$ 
 %containing all singletons
  and let $\K$ be the collection of all compact sets from $\I$. 
	
	If  $G$ is a $G_\delta$ set in $X$ such that $G\in \I$ and $\varphi:{X}\w Y$ is a continuous map described in Lemma \ref{compactification of zero-dim},
	% onto a compactum $Y$ such that $\varphi|G$ is a homeomorphism onto $\varphi(G)$ such that
	then $\bar{A}\in \K$ for any $A\sub G$ such that $\varphi|A$ extends to a homeomorphism onto $\bar{A}$.
	
	Consequently,	if $T\sub G$ is a $\kappa$-$\K$-Lusin set in $G$ of cardinality $\lambda$, where $\aleph_1\leq \kappa\leq\lambda\leq\co$,
	%such that $|T|=\lambda$ and $|T\cap A|<\kappa$ for %every $A\in {\K}$,
	then letting $E=\varphi(T)$ and $f=\varphi^{-1}|E:E\to {X}$, we obtain a continuous function on a set of cardinality $\lambda$, which is not uniformly continuous on any subset of $E$ of cardinality $\kappa$.
	
\end{proposition}

\begin{proof} Let us fix $A\sub G$ such that $\varphi|A$ extends to a homeomorphism
$\tilde{\varphi}$ between $\bar{A}$ and $\overline{\varphi(A)}$. Then the set
$\bar{A}\setminus G= {\tilde{\varphi}}^{-1}\bigl(\overline{\varphi(A)}\setminus \varphi(G)\bigr)$ is countable, so $\bar A\sub G\cup (\bar{A}\setminus G)\in \I$, hence $\bar{A}\in \K$.

The final assertion follows directly from Proposition \ref{C_8 from calK-Lusin}.

\end{proof}

As the following proposition shows, the above observation could be applied to various natural \s-ideals. Let us recall that un uncountable set $T$ in a Polish space $Y$ is a {\sl Lusin set in $Y$}, if $|T\cap D|<\aleph_1$ for every closed nowhere-dense subset of $Y$. 

\begin{proposition}\label{Solecki} 
 Let  $\I$ be a \s-ideal on ${X}$, let $\K$ be the collection of all compact sets from $\I$ and let us assume that $\I$ is not generated by $\K$ (i.e., there is a  set from $I$ which is not covered by any $F_\sigma$-set from $\I$).
 %containing all singletons
  Then, there exists a $G_\delta$-set in $X$ such that $G\in \I$ but no non-empty, relatively open set in $G$ is covered by an $F_\sigma$-set from $\I$. Consequently, every Lusin set $T$ in $G$ is $\aleph_1$-$\K$-Lusin in $G$ hence, it gives rise to a $C_8$-example (as described in Proposition \ref{zero-dim examples from ideals}). 
  \end{proposition} 
	\begin{proof}
		Let $B\in \I$ be  (a Borel) set not covered by any $F_\sigma$-set from $\I$. Then the existence of a $G_\delta$-set in $X$ such that $G\sub B$ but  $G$ is not covered by any $F_\sigma$-set from $\I$ follows from a theorem of Solecki (see \cite{sol}). By shrinking $G$, if necessary, we may assume that   $G$ has the desired properties. 
		
		It follows that if $K\in\K$, then $G\cap K$ is meager in $G$, and, consequently, $T\cap K$ is countable for any Lusin set $T$ in $G$. 
	\end{proof}

\begin{remark} 
A typical example of the situation described in Propositions \ref{zero-dim examples from ideals} and \ref{Solecki} is 	when $X$ is a copy of the Cantor in $\RR$ of positive Lebesgue measure, $\I$ is the \s-ideal of Lebesgue measure zero Borel sets in $X$ (then $\K$ is the family of  closed Lebesgue measure zero sets in $X$) and  $G$ is a dense copy of irrationals in $X$ of Lebesgue measure zero. 

Then the function $f=\varphi^{-1}|H:H\to \RR$, where $H=\varphi(G)$, is a homeomorphic embedding of $H$,  a copy of irrationals, to $\RR$, with the property that for every Lusin set $L$ of cardinality $\co$ in $H$, the function $f|L$ is not uniformly continuous on any uncountable subset of $L$. This provides an alternative proof of the theorem of Sierpiński that the existence of a Lusin set of cardinality $\co$ in $\PP$ implies $C_8$ (cf. \cite[proof of Th\'eor\`eme 6 on page 45]{Si}).
\end{remark}

\subsection{Infinite-dimensional spaces}
The assumption that the space $X$ is zero-dimensional  in Theorem \ref{theorem 1}  is essential, as demonstrated by the following result (CH in this theorem can be weakened to the assumption that no family of less than $\co$ meager sets covers $\RR$, cf. Remark \ref{cov(M)}).

\begin{theorem}\label{C_8 in inf dim}
	Assuming CH, there exists a set  $E\sub I^\NN$ of cardinality $\co$ such that 
	\begin{enumerate}
		\item there is a continuous function $f:E\to I^\NN$, which is not uniformly continuous on any subset of $E$ of cardinality $\co$,
		
		\item each continuous function $g:E\to \RR$ is constant on a subset of $E$  of cardinality $\co$.
	\end{enumerate}
\end{theorem} 

A key element of the proof of this theorem is a 
{\sl Henderson compactum} -- a compact metrizable infinite-dimensional space
whose all compact subsets of finite dimension are zero-dimensional, cf. \cite{vM}. 

More specifically, we shall need the following fact, where 
{\sl punctiform} sets are the sets without non-trivial subcontinua, cf. \cite[1.4.3]{eng-2}. 
%By a {\sl compactum} we mean a compact metrizable space.

\begin{lemma}\label{key lemma}
	There exists a non-empty punctiform $G_\delta$-set $M$ in a Henderson compactum $H\sub I^\NN$ such that for each $M'\sub M$ with ${\rm dim}\ (M\setminus M')\leq 0$, every continuous function $u:M'\to \RR$ is constant on a set of positive dimension.
\end{lemma}	
	
	A justification of this lemma is rather standard, but since we did not find convenient references, we shall give a proof to this effect at the end of this subsection. For now, taking this fact for granted, we shall proceed with a proof of our main result.
	
	\begin{proof}[Proof of Theorem \ref{C_8 in inf dim}]
		Let us adopt the notation of Lemma \ref{key lemma}, and let $\K$ be the collection of all compact zero-dimensional sets in the Henderson compactum $H$.
		
	Using CH, we inductively construct a $\co$-$\K$-Lusin set $T$ in $M$ such that $|T\cap L|=\co$ for every Borel set $L$ in $M$ with ${\rm dim}\ L>0$. To that end, we list all elements of $\K$ as $\langle K_\alpha:\alpha<\co\rangle$ and all Borel sets $L$ in $M$ with ${\rm dim}\ L>0$ as $\langle L_\alpha:\alpha<\co\rangle$, repeating each such set $L$ continuum many times. Then, we subsequently pick for $\xi<\co$
	$$
	p_\xi\in L_\xi\setminus \bigl(\bigcup_{\alpha<\xi} K_\alpha \cup \{p_\alpha:\alpha<\xi\}\bigr),
	$$
	using the fact that, under CH, the set in brackets is zero-dimensional, by the sum theorem for dimension zero (cf.  \cite[1.3.1]{eng-2}). Finally, we let
	$$
	T=\{p_\xi:\xi<\co\}.%\quad\hbox{and}\quad E=\varphi(T).
	$$	
	
	In particular, $T$ meets each Borel set in $M$ of positive dimension in continuum many points.
		\ms 
		
			{\bf Claim 1.} Each continuous function $w:T\to \RR$ is constant on a set of cardinality $\co$.
			
			\ms 
			
	Indeed, $w$ can be extended to a continuous function $\tilde{w}:\tilde{T}\to \RR$ over a $G_\delta$-set $\tilde{T}$ in $M$ containing $T$. The Borel set $M\setminus \tilde{T}$ is disjoint from $T$, hence ${\rm dim}\ (M\setminus \tilde{T})\leq 0$. Setting in Lemma \ref{key lemma} $u=\tilde{w}$ and $M'=\tilde{T}$, we get $r\in\RR$ such that ${\rm dim}\ u^{-1}(r)>0$. Being a Borel set in $M$, $u^{-1}(r)$ intersects $T$ in a set of cardinality $\co$, i.e., $|w^{-1}(r)|=\co$.
	
	\ms 
	
	Let us now apply a counterpart of Lemma \ref{compactification of zero-dim} for infinite-dimensional spaces to the following effect (cf. the proof of \cite[Lemma 5.3.1]{eng-2}): there exists a continuous surjection $\varphi: H\to Y$ onto a compact metrizable space $Y$ such that $\varphi|M$ is a homeomorphism onto $\varphi(M)$, $\varphi(H\setminus M)\cap \varphi(M)=\emptyset$ and $Y\setminus \varphi(M)$ is a countable union of finite-dimensional compact sets.
	
	\ms 
	
		{\bf Claim 2.} The set $E=\varphi(T)$ satisfies the assertion of Theorem \ref{C_8 in inf dim}.

	\ms 

		Since $E$ is homeomorphic to $T$, Claim 1 shows that $(2)$ in Theorem \ref{C_8 in inf dim} is satisfied.
		
		To check also assertion $(1)$ in this theorem, we shall make sure that the function
		$$
		f= \varphi^{-1}|E: E\to H
		$$
		is not uniformly continuous on any set of cardinality $\co$.
		
		Since $T$ is a $\co$-$\K$-Lusin set in $M$, by Proposition \ref{C_8 from calK-Lusin} it is enough to verify that for any non-empty $A\sub M$, whenever $\varphi|\overline{A}$ is an embedding (the closure of $A$ is in ${H}$), then $\bar{A}\in \K$, i.e., ${\rm dim}\ \bar{A}=0$. 
		
		To check this, 	let  us note that 
		$$
		\varphi(\bar{A}\setminus M)=\varphi(\bar{A})\setminus\varphi(M)\sub Y\setminus\varphi(M)
		$$
		is a countable union of compact finite-dimensional sets, and so is $\bar{A}\setminus M$. Since $H$ is a Henderson compactum, $\bar{A}\setminus M$ is a countable union of zero-dimensional compact  sets, hence
		${\rm dim} (\bar{A}\setminus M)=0$ by the sum theorem (cf. \cite[1.4.5]{eng-2}). By the enlargement theorem, cf. \cite[1.5.11]{eng-2}, $\bar{A}\setminus M$ can be enlarged to a zero-dimensional $G_\delta$-set $L$ in $H$. The set $\bar{A}\setminus L$ is a $\sigma$-compact subset in the punctiform space $M$, hence again by the sum theorem, ${\rm dim}\ (\bar{A}\setminus L)\leq 0$ (cf. \cite[1.4.5]{eng-2}). By the addition theorem, cf. \cite[1.5.10]{eng-2},   ${\rm dim}\  \bar{A}\leq 1$, and $H$ being a Henderson compactum, ${\rm dim}\ \bar{A}=0$.
		
	\end{proof}

In the assumptions of Theorem \ref{C_8 in inf dim}, CH can be weakened to the assertion  (usually denoted by ${\rm cov}(\M)=\co$, cf. \cite{j-w}) that no family of less than $\co$ meager sets covers $\RR$. The only change in the proof requires checking that the inductive definition of the sequence $\langle p_\xi:\xi<\co\rangle$ is correct. More precisely, the following is true.

\begin{remark}\label{cov(M)}
	Assuming ${\rm cov}(\M)=\co$, the union of any family $\K$ of compact zero-dimensional sets in $I^\NN$, with $|\K|<\co$, is zero-dimensional.
\end{remark}

\begin{proof}
	Let	$\K$ be a family of compact zero-dimensional sets in $I^\NN$, with $|\K|<\co$. For any $K\in\K$, let us consider 
	\begin{enumerate}
		\item[(1)] $
		\U_K=\{f\in C(I^\NN,\RR):\ f^{-1}(0)\cap K=\emptyset\}.
		$
	\end{enumerate}
	
	Then $\U_K$ is open and dense in the function  space $C(I^\NN,\RR)$ and this space being perfect Polish, the assumption ${\rm cov}(\M)=\co$ guarantees that the set 
	\begin{enumerate}
		\item[(2)] $
		\G=\bigcap\{\U_K: K\in\K\}
		$
	\end{enumerate}
	is dense in $C(I^\NN,\RR)$ (cf. \cite[8.32]{ke}).
	
		Let us recall (cf. \cite{eng-2}) that a closed set $L$ in a topological space $Z$ {\sl  separates the space $Z$ between  sets $A_1,\ A_2\sub Z$},  if $Z\setminus L=U_0\cup U_1$, where $U_0$, $U_1$ 
	are open, disjoint and $A_i\sub U_i$, for $i=0,1$. 
	If $A_1$ and $A_2$ are singletons, then we say that $L$ separates $Z$ between the respective points.
	
	Let $F= \bigcup \K$. To prove that ${\rm dim}\ F =0$, it is enough to show that  for any pair of disjoint compact sets $A,\ B$ in $I^\NN$, $I^\NN$ can be separated between $A$ and $B$ by a closed set disjoint from  	$F$. Indeed, let $x\in F$ and let $U$ be a relatively open subset of $F$ with $x \in U$ and $F\setminus U\neq\emptyset$. Then $U=V\cap F$ for an open $V$ in $I^\NN$ with $B= I^\NN\setminus V\neq\emptyset$, and separating $I^\NN$ between $A=\{x\}$ and $B$ by a closed set in $I^\NN$, disjoint from $F$, leads to disjoint open sets $V_1$, $V_2$ in $I^\NN$ with $x \in V_1\cap  F \sub U$, $V_2\cap F\neq\emptyset$, and $ F = (V_1\cap  F) \cup (V_2\cap   F)$. Thus, $V_1\cap F$ is a relatively clopen in $F$ neighbourhood of $x$ contained in $U$.
	
	%, $B\sub V_2$ and $X \setminus L = V_1  \cup V_2$. But then $ \bigcup \K = (V_1\cap  \bigcup \K) \cup (V_2\cap   \bigcup \K)$ is a partition of $ \bigcup \K$ into open sets, so $V_1\cap  \bigcup \K$ is clopen in $ \bigcup \K$, contains $x$ and moreover $V_1 \cap  \bigcup \K \sub U$ since $V_1$ is disjoint from $B$, so it is contained in $V$. 

	So let $A,\ B$ be disjoint compact sets  in $I^\NN$. We can pick $f\in \G$ such that $f(A)\sub (-\infty,0)$,
	$f(B)\sub (0,+\infty)$, as such functions form an open, non-empty set in $C(I^\NN,\RR)$. Then  $L=f^{-1}(0)$ is a closed set in $I^\NN$ separating  $I^\NN$ between $A$ and $B$ with the property that $L\cap K=\emptyset$ for all $K\in\K$, cf. (1) and (2).

\end{proof}

\begin{proof}[Proof of Lemma \ref{key lemma}]
	% Let us recall  that a closed set $L$ in a topological space $Z$ {\sl separates  $Z$ between  points $z_0,\ z_1\in Z$},  if $Z\setminus L=U_0\cup U_1$, where $U_0$, $U_1$ are open, disjoint and $z_i\in U_i$, for $i=0,1$. 

	 Let us fix a Henderson continuum $K$ in $I^\NN$. We shall consider on $I^\NN$ the metric assigning to points $(s_0, s_1,\ldots),\ (t_0, t_1,\ldots)$ the distance $\sum\limits_i 2^{-i}|s_i-t_i|$.
	 
	 \ms 
	 	{\bf Claim 1.} There exists a punctiform $G_\delta$-set $S$ in $K$ and distinct points $a, b\in S$ such that each relatively closed set in $S$ separating $S$ between $a$ and $b$ has dimension at least 2.
	 	
	 	\ms 
	 	
	 	Indeed, since ${\rm dim}\ K \geq 4$, $K$ contains a punctiform $G_\delta$-set $W$ with ${\rm dim}\ W \geq 3$.
	 	
	 	This   theorem goes back to Mazurkiewicz, and can be justified  by the reasoning in the proof of \cite[Theorem 3.9.3]{vM} applied to a pair of disjoint compact sets $A$ and $B$ in $K$ such that for any set $N$ in $K$ with ${\rm dim}\ N \leq 2$ there is a continuum in $K$ joining $A$ and $B$ and missing $N$ (cf. \cite[Theorem 4.2]{r-s-w}), and to a continuous map $\pi:K\to[-1,1]$ sending $A$ to $-1$ and $B$ to $1$.
	 	
	 	\ss 
	 	
	 	Now, since ${\rm dim}\ W \geq 3$, there is $a\in W$ such that all sufficiently small neighbourhoods of $a$ in $W$ have boundaries of dimension $\geq$ 2. An argument in \cite[proof of Theorem 8, p. 172]{kur-II}
	 	gives a point  $b\in K\setminus \{a\}$ such that for $S=W\cup\{b\}$, every relatively closed set in $S$ separating $S$ between $a$ and $b$ has dimension $\geq$ 2. Since $S$ is a punctiform $G_\delta$-set in $K$, it satisfies Claim 1. 
	 	
	 	 \ms 
	 	 
	 	{\bf Claim 2.} There are embeddings $h_n: I^\NN\to I^\NN$, $n=1,2,\ldots$, such that, letting 
	 	\begin{enumerate}
	 		\item[] $M_0=S\times \{0\}$,  $M_n=h_n(S)\times \{\frac{1}{n}\}$\ \hbox{for}\  $n=1,2,\ldots$,
	 		\item[] $M=\bigcup\limits_{n=0}^\infty M_n\subseteq I^\NN\times I$\ \hbox{and}\ 
	 		$H=\overline{M}$ -- the closure of $M$ in $I^\NN\times I$,
	 	\end{enumerate}
	 	one obtains sets $M\sub H$ satisfying the assertion of Lemma \ref{key lemma}.
	 	
	 	\ms 
	 	
	 	We adopt the notation from Claim 1. Since $K$ is a compact, connected set in $I^\NN$, we can find open connected neighbourhoods $U_1\supseteq  U_2\supseteq\ldots$   of $K$  such that each open neighbourhood of $K$ in $I^\NN$ contains some $U_n$.
	 	
	 	Let $D$ be a countable set dense in $S$, and let $\langle (c_n,d_n):\ n\in\NN\rangle$
	 	be an enumeration of all ordered pairs of distinct points from $D$, such that each such pair appears in the sequence infinitely many times.
	
		Let us fix $n>0$.  We shall define an embedding $h_n: I^\NN\to U_n$ such that the distance from $h_n(a)$ to $c_n$  and the distance from $h_n(b)$ to $d_n$ is less than $\frac{1}{n}$.
	
	 To that end, let us notice that $U_n$ is arcwise connected, being an open,  connected set in $I^\NN$
	 (cf. \cite[Proposition 12.25]{suth}), and hence there is a continuous $f: I^\NN\to U_n$ with $f(a)=c_n$, $f(b)=d_n$.
	 Let $m\geq n$ be large enough to ensure that for any $x,y \in I^\NN$, 
	  whenever the first $m$ coordinates of $x$  and $f(y)$  
	  coincide, then $x\in U_n$. Now, denoting by  $p_{j}: I^\NN\to I$ the  projection  onto  $j$'th coordinate, 
	 we define the embedding $h_n$ by $p_i(h_n(x))=p_i(f(x))$ for $i\leq m$ and 
	$p_{m+i}(h_n(x))=p_i(x)$ for $i=1,2,\ldots$.
	
	Having defined the embeddings $h_n$, we shall check that the set $M$ in Claim 2 satisfies the assertion of Lemma \ref{key lemma}.
	
	Clearly, $M$ is a punctiform $G_{\delta}$-set in $I^\NN\times I$. 
	
	Since each neighbourhood of $K$ in $I^\NN$ contains all but finitely many $h_n(K)\sub \overline{U_n}$, the set 
	$$
	K\times \{0\} \cup \bigcup\limits_{n=1}^\infty\Bigl(h_n(K)\times \Bigl\{\frac{1}{n}\Bigr\}\Bigr)
	$$ 
	contains $H=\overline{M}$, and since it  is clearly a Henderson compactum containing $M$, so is $H$.
	
	Let $M'\sub M$ satisfy ${\rm dim}\ (M\setminus M')\leq 0$ and let $u:M'\to \RR$ be continuous.
	
	Since ${\rm dim}\ (M_0)= {\rm dim}\ (S)\geq 2$, the addition theorem yields that ${\rm dim}\ (M'\cap M_0)\geq 1$. Therefore, if $u$ is constant on $M'\cap M_0$, we are done.
	
	Let us assume that this is not the case. Then for some $x,y\in M'\cap M_0$, $u(x)\neq u(y)$, hence for $r= \frac{u(x)+u(y)}{2}$, $u^{-1}(r)$ is a relatively closed set in $M'$ which separates $M'$ between $x$ and $y$. One can extend $u^{-1}(r)$ to a closed set $L$ in $H$ such that $L$ separates $H$ between $x$ and $y$, cf. \cite[\S 21, XI, Theorem 2, p. 226]{kur-I}. Let us pick $c,d\in D$, close enough to $x$ and $y$, respectively, so that $L$ separates $H$ between $(c,0)$ and $(d,0)$, and using the fact that $(c,d)=(c_n,d_n)$ for infinitely many $n$'s and the distances from $h_n(a)$ to $c_n$, and  from $h_n(b)$ to $d_n$ tend to zero, let us pick $n$ such that 
	$L$ also separates $H$ 
	between $(h_n(a),\frac{1}{n})$ and $(h_n(b),\frac{1}{n})$. In effect, since $x\mapsto (h_n(x),\frac{1}{n})$ is a homeomorphism of $S$ onto $M_n$, and each relatively closed set separating $S$ between $a$ and $b$ has dimension $\geq$ 2, we infer that ${\rm dim}\ (M_n \cap L)\geq 2$. Since ${\rm dim}\ (M_n \setminus M')\leq 0$, applying again the addition theorem, we conclude that ${\rm dim}\ (M' \cap L)\geq 1$.
	
	Finally, since $M' \cap L = u^{-1}(r)$,  $u$ is constant on a set of positive dimension.

\end{proof}

\section{A characterization of complete metrizability}\label{sec:5}

A key step in proving that statement $C_9(\lambda,\kappa)$ is equivalent to the existence of a  $\kappa$-K-Lusin set $E$ of cardinality $\lambda$ in $\PP$ (cf.  \cite[Theorem 2.3]{C_8_vs_C_9}) is Theorem 2.1 in \cite{C_8_vs_C_9} (closely related to a theorem of Sierpiński, cf.  \cite[Remark 2.2]{C_8_vs_C_9})  stating that for any Polish space $X$ there is a sequence 
$f_0\geq f_1\ldots$ of continuous functions $f_n:X\w I$ which converges to zero pointwise but does not converge uniformly on any set with non-compact closure in $X$.  

In fact, the existence of such a function sequence characterizes 
complete metrizability  of a separable metrizable space $X$.

\begin{theorem}\label{characterization_of_completeness}
	A separable metrizable space $X$ is completely metrizable if and only if there is a sequence $f_0\geq f_1\geq \ldots$ of continuous functions $f_n:X\w I$ which converges to zero pointwise but does not converge uniformly on any %closed non-compact
	 set with non-compact closure in $X$.  
\end{theorem}

\begin{proof}
	In view of \cite[Theorem 2.1]{C_8_vs_C_9} we only have to prove the "if" part of the above equivalence. So let $(X,d)$ be a separable metric space.
	Let	$f_n:X\w I$, $n\in\NN$, be a sequence  of continuous functions with $f_0\geq f_1\geq \ldots$,  which converges to zero pointwise but does not converge uniformly on any set with non-compact closure in $X$. 
	
	 Let $(\hat{X}, \hat{d})$ be the completion of $(X,d)$; clearly, the space $\hat{X}$ is Polish.  There exists a $G_\delta$-set $\tilde{X}$ in $\hat{X}$, containing $X$ such that each $f_n$ extends to a continuous function $\tilde{f_n}:\tilde{X} \to I$. Since $X$ is dense in $\tilde{X}$, we have $\tilde{f_0}\geq \tilde{f_1}\geq \ldots$ and in particular, the set
$$
G=\{x\in \tilde{X}:\lim_{n\to\infty}\tilde{f_n}(x)=0 \}
$$
is $G_\delta$ in $\tilde{X}$.

Since $X\sub G$, it remains to make sure that $G\sub X$, to conclude that $X=G$ is completely metrizable as a $G_\delta$-subset of the (Polish) space $\hat{X}$.

So let   $c\in G$ and let us pick $x_k\in X$, $k\in\NN$, such that $\lim_{k\to\infty}x_k=c$. 

Let $K=\{c\}\cup \{x_k:k\in\NN\}$. Then $K$ is compact, so the sequence $(\tilde{f_n})_{n\in\NN}$ converges uniformly on $K$ (this is a special instance of the Dini theorem \cite[Lemma 3.2.18]{eng-1}). It follows that the sequence $(f_n)_{n\in\NN}$ converges uniformly on $K\cap X$, a closed set in $X$ which therefore, by the assumed property of $(f_n)_{n\in\NN}$, is compact. Since the sequence $(x_n)_{n\in\NN}$ converges to $c$ and all $x_n$ are elements of $K\cap X$, so is $c$. In particular, $c\in X$. 

\end{proof}

With the help of Theorem \ref{characterization_of_completeness} we shall establish the following more general result (for terminology see \cite{eng-1}).

\begin{theorem}\label{Cech_complete}
	Let $X$ be a Hausdorff space. Then $X$ is a \v{C}ech-complete Lindel\"{o}f space if and only if there is a sequence $f_0\geq f_1\geq\ldots$ of continuous functions $f_n:X\to I$ converging pointwise to zero but not converging uniformly on any closed non-compact set in $X$.
\end{theorem}

\begin{proof}
	First, let $X$ be a \v{C}ech-complete Lindel\"{o}f space. By a theorem of Frol\'ik, it follows that
	there is a perfect map $p:X\to Y$ onto a Polish space $Y$, cf. \cite[5.5.9]{eng-1}. Let us recall that for a Hausdorff space $X$ and a metrizable space $Y$ this means, cf. \cite[Theorems 3.7.2 and 3.7.18]{eng-1}, that $p$ is  a continuous, closed mapping and the inverse image of every compact subset of $Y$ is compact. It is  straightforward to check that if functions $f_n:Y\to I$, $n\in\NN$, with $f_0\geq f_1\geq\ldots$,  satisfy the assertion of Theorem \ref{characterization_of_completeness}, the functions $f_n\circ p:X\to I$, $n\in\NN$, have the required properties.
	
	\ss 
	
	Next, given a sequence  $f_0\geq f_1\geq\ldots$ of functions $f_n:X\to I$, described in the theorem, let us consider the diagonal map 
	\begin{enumerate}
		\item[(1)] $F=(f_0,f_1,\ldots):X\to I^\NN$ and let $Y=F(X)$.
	\end{enumerate}

	 Let $p_n:I^\NN\to I$ denote the projection onto the $n$'th coordinate and let, cf. (1), for $n\in\NN$,
	 \begin{enumerate}
	 	\item[(2)] $g_n=p_n|Y:Y\to I$.
	 \end{enumerate}
 
	 Note that, since $f_n=g_n\circ F$, the properties of the sequence $(f_n)_{n\in\NN}$ are passed to the sequence $(g_n)_{n\in\NN}$, namely:
	\begin{enumerate}
	\item[(3)] $g_0\geq g_1\geq\ldots$ is a sequence of continuous functions  converging pointwise to zero,
	 \end{enumerate}
 and for any closed set $A$ in $Y$,
 \begin{enumerate}
 	\item[(4)] if $(g_n)_{n\in\NN}$ converges on $A$ uniformly, then $A$ is compact.
 \end{enumerate}

 To see (4), let us consider $B=F^{-1}(A)$. Then, as $f_n=g_n\circ F$, $(f_n)_{n\in\NN}$ converges on $B$ uniformly, hence $B$ is compact, and so is $A= F(B)$.
 
 By Theorem \ref{characterization_of_completeness},
 \begin{enumerate}
 	\item[(5)] $Y$ is a $G_\delta$-set in $I^\NN$.
 \end{enumerate} 

To complete the proof, it is enough to show that 
\begin{enumerate}
	\item[(6)] the mapping $F$ is perfect.
\end{enumerate} 

 Indeed, if (6) holds, then since by (5), $Y$ is a \v{C}ech-complete Lindel\"{o}f space, so  is its perfect preimage $X$, cf. \cite[Theorems 3.8.9 and 3.9.10]{eng-1}.
 
 To prove (6), it suffices to check that 
 %\begin{enumerate}
 %	\item[(7)]
  for every compact $K\sub Y$ the inverse image  $F^{-1}(K)$ is compact, cf.  \cite[Theorem 3.7.18]{eng-1}.
 %\end{enumerate} 

So let $K\sub Y$ be compact. Then, again by the Dini theorem, the sequence $(g_n)_{n\in\NN}$ converges uniformly on $K$, and since $f_n=g_n\circ F$, cf. (2), the sequence $(f_n)_{n\in\NN}$ converges uniformly on $F^{-1}(K)$. By the assumed property of this sequence, $F^{-1}(K)$ is compact, completing the proof.

\end{proof}

\bibliographystyle{amsplain}

\end{document}